\documentclass[12pt]{article}
\usepackage{amsmath,amssymb,latexsym}
\usepackage{enumerate}
\usepackage{amsthm}
\usepackage{graphicx}
\usepackage{epsfig}

\newcommand\NN{\mathbb{N}}

\newcommand\ZZ{\mathbb{Z}}

\DeclareMathOperator\dist{dist}

\newtheorem{theorem}{Theorem}[section]

\newtheorem{corollary}[theorem]{Corollary}

\newtheorem{remark}[theorem]{Remark}
\newtheorem{proposition}[theorem]{Proposition}

\numberwithin{equation}{section}

\newcommand{\address}{Address: Department of Mathematics, University of North Texas, 1155 Union Circle \#311430, Denton, TX 76203-5017, USA; E-mail: NathanDalaklis@my.unt.edu, kiko.kawamura@unt.edu,TobeyMathis@my.unt.edu, MichalisPaizanis@my.unt.edu}

\title{The partial derivative of Okamoto's functions with respect to the parameter}
\author{Nathan Dalaklis, Kiko Kawamura, Tobey Mathis, Michalis Paizanis \\University of North Texas \footnote{\address}}

\begin{document}

\maketitle

\begin{abstract}
The differentiability of the one parameter family of Okomoto's functions as functions of $x$ has been analyzed extensively since their introduction in 2005. As an analogue to a similar investigation, in this paper, we consider the partial derivative of Okomoto's functions with respect to the parameter $a$. We place a significant focus on $a = 1/3$ to describe the properties of a nowhere differentiable function $K(x)$ for which the set of points of infinite derivative produces an example of a measure zero set with Hausdorff dimension $1$.

\end{abstract}

%\begin{classification}
%26A27, 26A30 (primary); 28A78, 11A63 (secondary)
%\end{classification}

%\begin{keywords}
%Continuous nowhere differentiable function, singular function, Cantor function, partial derivative, ternary expansion, infinite derivative.
%\end{keywords}

\section{Introduction}

In 1957, De Rham~\cite{Rham} studied a unique continuous solution $L_a(x)$ of the following functional equation. 
\begin{equation} 
\label{FE:L_a}
        L_{a}(x)=
        \begin{cases}
                a L_{a}(2x), 
                        & \qquad 0 \leq x \leq \tfrac12, \\
                (1-a)L_{a}(2x-1)+a, 
                        & \qquad \tfrac12 \leq x \leq 1, \\
        \end{cases}
\end{equation}
where $0<a<1, a \neq 1/2$. Also known as Lebesgue's singular function, $L_a(x)$ is strictly increasing and has a derivative equal to zero almost everywhere. From \eqref{FE:L_a}, it is clear that $L_a$ is self-affine: The portions of the graph above the intervals $[0,1/2]$ and $[1/2,1]$ are affine images of the whole graph, contracted horizontally by $1/2$ and vertically by $a$ and $1-a$, respectively. 

%%%%%%%figure0%%%%%%%%%%%
\begin{figure}
\begin{center}
\includegraphics[height=5cm,width=5cm, bb=100 0 700 550]{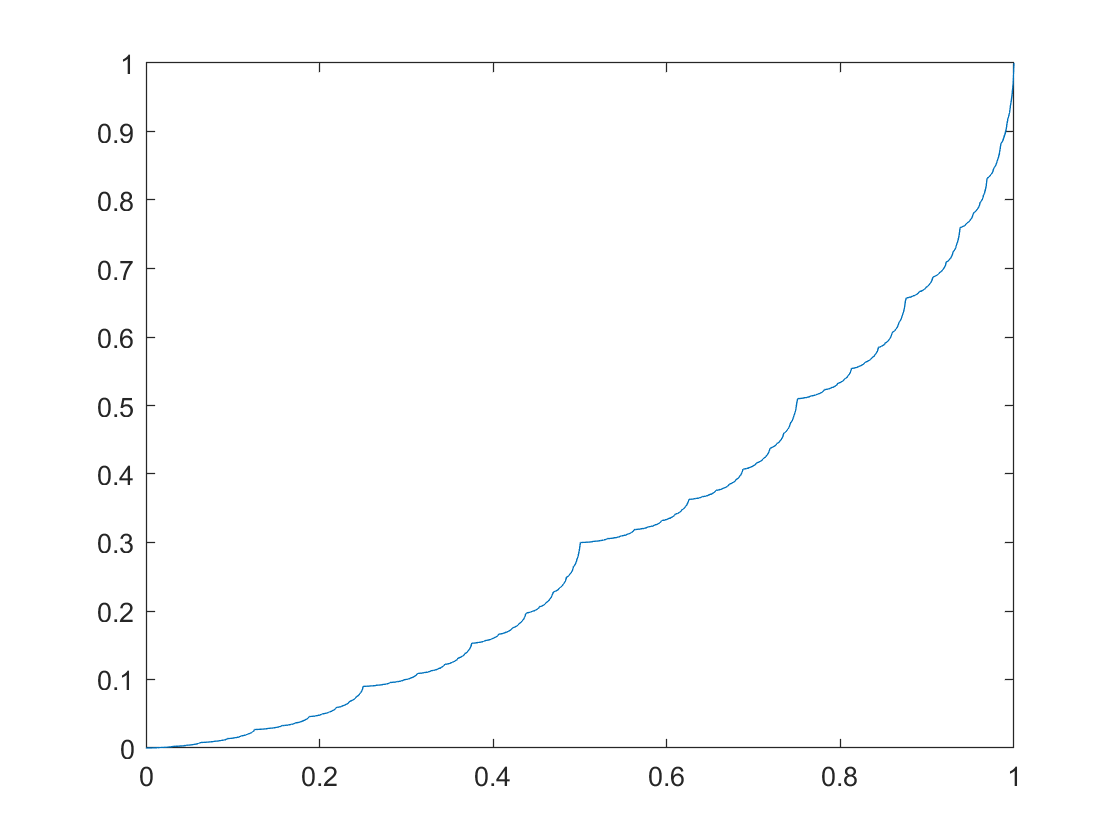} 
\qquad \qquad 
\includegraphics[height=5cm,width=5cm, bb=100 0 700 550]{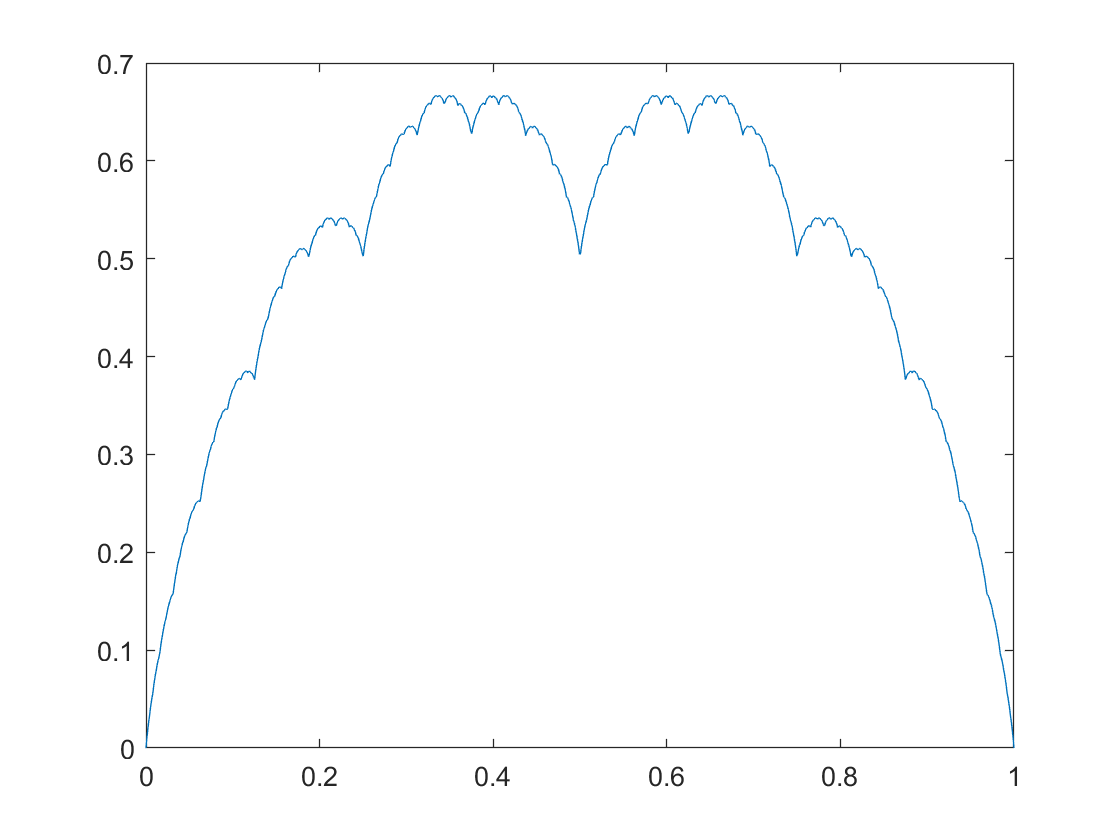}
\end{center}
\caption{Graph of $L_{a=1/3}(x)$ and $T(x)$} 
\label{figure0}
\end{figure}
%%%%%%%%%%%%%%%%%%%%%%%%%%%
Another famous continuous nowhere differentiable function, namely that of Takagi~\cite{Takagi} was introduced in 1903. It is defined by
\begin{equation}
\label{expression-takagi}
T(x)=\sum_{n=0}^\infty \frac{1}{2^n}\phi(2^n x),
\end{equation}
where $\phi(x)=\dist(x,\ZZ)$, the distance from $x$ to the nearest integer. Functions, like Takagi's function, that are not sufficiently smooth or regular have often been ignored as `pathological' and not worthy of study. Prior to 1984, there was no known relationship between the functional equation, $L_a(x)$, and Takagi's function, $T(x)$. However, that year Hata and Yamaguti~\cite{Hata-Yamaguti2} showed the following beautiful relationship between $L_a(x)$ and $T(x)$:
\begin{equation}
    \left.\frac{\partial L_{a}(x)}{\partial  a}\right|_{a=1/2}=2T(x).
\end{equation}

\noindent A generalization of this result of Hata and Yamaguti was considered by Sekiguchi and Shiota~\cite{Shiota}. They computed the $k$-th partial derivative of $L_{a}(x)$ with respect to the real parameter $a$ and applied it to digital sums problems~\cite{Okada-Sekiguchi-Shiota}. Also, Tasaki, Antoniou and Suchanecki pointed out applications of Hata and Yamaguti's results in physics ~\cite{Tasaki}. Later, Kawamura expanded Hata and Yamaguti's results to classify all self-similar sets determined by two contractions in the plane~\cite{Kawamura}. Note that both Lebesgue's singular function and Takagi's nowhere differentiable function are expressed by binary expansions of $x$ on the interval $[0,1]$. 

\medskip\newpage

In 2005, H. Okamoto \cite{Okamoto} introduced another one-parameter family of self-affine functions $\{F_a: 0<a<1\}$, which are expressed by ternary expansions of $x$ on the interval $[0,1]$. He defined $F_a$ as follows. Let $f_0(x)=x$, and inductively, for $n=0,1,2,\dots$, let $f_{n+1}$ be the unique continuous function which is linear on each interval $[j/3^{n+1},(j+1)/3^{n+1}]$ with $j\in\ZZ$ and satisfies, for $k=0,1,\dots,3^n-1$, the equations
\begin{gather*}
f_{n+1}(k/3^n)=f_n(k/3^n), \qquad f_{n+1}\big((k+1)/3^n\big)=f_n\big((k+1)/3^n\big),\\
f_{n+1}\big((3k+1)/3^{n+1}\big)=f_n(k/3^n)+a\left[f_n\big((k+1)/3^n\big)-f_n(k/3^n)\right],\\
f_{n+1}\big((3k+2)/3^{n+1}\big)=f_n(k/3^n)+(1-a)\left[f_n\big((k+1)/3^n\big)-f_n(k/3^n)\right].
\end{gather*}
The sequence $(f_n)$ defined above converges uniformly on $[0,1]$. Let $F_a:=\lim_{n\to\infty}f_n$, then $F_a$ is a continuous function from the unit interval $[0,1]$ onto itself.  By changing the parameter $a$, one can produce some interesting examples: Perkins' nowhere differentiable function~\cite{Perkins} (when $a=5/6$), Bourbaki-Katsuura's function ~\cite[p.~35, Problem 1-2]{Bourbaki},~\cite{Katsuura} (when $a=2/3$) and Cantor's Devil's staircase function (when $a=1/2$). Okamoto and Wunsch~\cite{OkaWunsch} proved that $F_a$ is singular when $0<a\leq 1/2$ and $a\neq 1/3$. 
\begin{figure}
\begin{center}
\begin{picture}(360,150)(0,15)
\put(30,20){\line(1,0){126}}
\put(30,20){\line(0,1){126}}
\put(30,146){\line(1,0){126}}
\put(156,20){\line(0,1){126}}
\put(27,15){\makebox(0,0)[tl]{$0$}}
\put(72,18){\line(0,1){4}}
\put(62,15){\makebox(0,0)[tl]{$1/3$}}
\put(114,18){\line(0,1){4}}
\put(104,15){\makebox(0,0)[tl]{$2/3$}}
\put(154,15){\makebox(0,0)[tl]{$1$}}
\put(17,25){\makebox(0,0)[tl]{$0$}}
\put(28,104){\line(1,0){4}}
\put(16,107){\makebox(0,0)[tl]{$a$}}
\put(28,62){\line(1,0){4}}
\put(-2,67){\makebox(0,0)[tl]{$1-a$}}
\put(19,150){\makebox(0,0)[tl]{$1$}}
\put(30,20){\line(1,2){42}}
\put(72,104){\line(1,-1){42}}
\put(114,62){\line(1,2){42}}
\put(60,122){\makebox(0,0)[tl]{$f_1$}}
\thicklines
%\dottedline{4}(30,20)(156,146)
\thinlines
\put(230,20){\line(1,0){126}}
\put(230,20){\line(0,1){126}}
\put(230,146){\line(1,0){126}}
\put(356,20){\line(0,1){126}}
\put(227,15){\makebox(0,0)[tl]{$0$}}
\put(272,18){\line(0,1){4}}
\put(262,15){\makebox(0,0)[tl]{$1/3$}}
\put(314,18){\line(0,1){4}}
\put(304,15){\makebox(0,0)[tl]{$2/3$}}
\put(354,15){\makebox(0,0)[tl]{$1$}}
\put(217,25){\makebox(0,0)[tl]{$0$}}
\put(228,104){\line(1,0){4}}
\put(216,107){\makebox(0,0)[tl]{$a$}}
\put(228,62){\line(1,0){4}}
\put(198,67){\makebox(0,0)[tl]{$1-a$}}
\put(219,150){\makebox(0,0)[tl]{$1$}}
\put(230,20){\line(1,4){14}}
\put(244,76){\line(1,-2){14}}
\put(258,48){\line(1,4){14}}
\put(272,104){\line(1,-2){14}}
\put(286,76){\line(1,1){14}}
\put(300,90){\line(1,-2){14}}
\put(314,62){\line(1,4){14}}
\put(328,118){\line(1,-2){14}}
\put(342,90){\line(1,4){14}}
\put(260,122){\makebox(0,0)[tl]{$f_2$}}
\thicklines
%\dottedline{4}(230,20)(272,104)
%\dottedline{4}(272,104)(314,62)
%\dottedline{4}(314,62)(356,146)
\end{picture}
\end{center}
\caption{The first two steps in the construction of $F_a$}
\label{fig:construction}
\end{figure}
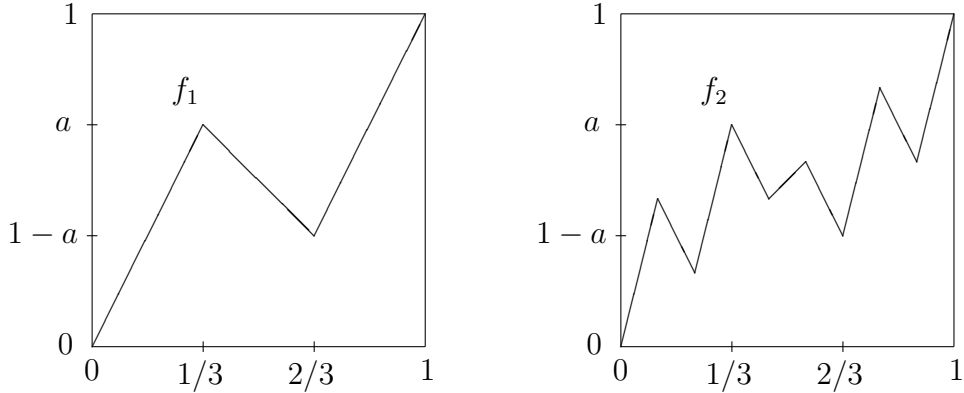

Let $a_0\approx .5592$ be the unique real root of $54a^3-27a^2=1$. Okamoto \cite{Okamoto} showed that 
\begin{enumerate}
\item $F_a$ is nowhere differentiable if $2/3\leq a<1$;
\item $F_a$ is nondifferentiable at almost every $x\in[0,1]$ but differentiable at uncountably many points if $a_0<a<2/3$; and 
\item $F_a$ is differentiable almost everywhere but nondifferentiable at uncountably many points if $0<a<a_0$. 
\end{enumerate}
Okamoto left the case $a=a_0$ open, but Kobayashi~\cite{Kobayashi} later showed, using the law of the iterated logarithm, that $F_{a_0}$ is nondifferentiable almost everywhere. Later, Allaart~\cite{Allaart} investigated the set of points at which $F_a$ has an {\em infinite} derivative and showed that this set has an interesting connection with $\beta$-expansions.

The graph of $F_a$ is self-affine: The portions of the graph above the intervals $[0,1/3]$, $[1/3,2/3]$ and $[2/3,1]$ are affine images of the whole graph, contracted horizontally by $1/3$ and vertically by $a$, $|2a-1|$ and $a$, respectively. The middle one is also reflected vertically when $a>1/2$. In fact, $F_a(x)$ satisfies the following functional equation:
\begin{equation}
\label{FE:F_a}
        F_a(x)=
        \begin{cases}
                aF_a(3x),
                  & \qquad 0 \leq x \leq 1/3, \\
                (1-2a)F_a(3x-1)+a,
                  & \qquad 1/3 \leq x \leq 2/3, \\
								aF_a(3x-2)+(1-a),
								  & \qquad 2/3 \leq x \leq 1, \\
       \end{cases}
\end{equation}
where $0<a<1$. 
As a result, the box-counting dimension of the graph of $F_a$ follows from Example 11.4 in \cite{Falconer}. It is $1$ if $a\leq 1/2$, and $1+\log_3(4a-1)$ if $a>1/2$. 

\medskip

Notice that $F_{1/3}(x)=x$ and $L_{1/2}(x)=x$. The main purpose of this article is to investigate the partial derivative of $F_a(x)$ as an analogy of Hata and Yamaguti's result. Let 
\begin{equation*}
K(x):=\left.\frac{\partial F_{a}(x)}{\partial  a}\right|_{a=1/3}.
\end{equation*}

\medskip
\noindent In Section $2$, we derive a simpler expression for $K(x)$:
\begin{equation}
\label{expression-K2}
K(x)=\sum_{n=0}^{\infty}\frac{1}{3^n}\Phi(3^{n}x),
\end{equation}
where 
\begin{equation*}
\Phi(x):=
\begin{cases}
	3x, & \qquad 0 \leq x \leq 1/3, \\
  3(1-2x), & \qquad 1/3 \leq x \leq 2/3, \\
	3(x-1), & \qquad 2/3 \leq x \leq 1.
\end{cases}
\end{equation*}

\noindent This particular expression helps us prove that $K(x)$ is continuous but nowhwere differentiable, just like Takagi's function. 

\medskip

In Section $3$, we give a complete description of the set of points at which $K(x)$ has an infinite derivative in terms of the ternary expansion of $x$. That is, for $x\in[0,1]$, the {\em ternary expansion} of $x$ is the sequence $\varepsilon_1,\varepsilon_2,\dots$ defined by 
$$x=\sum_{k=1}^\infty \varepsilon_k/3^k, \mbox{where } \varepsilon_k\in\{0,1,2\}$$ 
for all $k$. If $x$ has two ternary expansions we take the one ending in all $0$'s, except when $x=1$, in which case we take the expansion ending in all $2$'s. 
Then, for $n\in\NN$, let $I_1(n)$ be the number of 1's occurring in the first $n$ ternary digits of $x$. That is,
\begin{equation*}
I_1(n):=\#\{j \leq n: \varepsilon_j=1\}.
\end{equation*}
The definition of which allows us to state our main Theorem.
\begin{theorem}
\label{main-theorem}
$K^{'}(x)=\pm \infty$ if and only if $x$ satisfies the following condition: 
\begin{equation}
\label{eq:iff condition}
n-3I_1(n) \to \pm \infty \qquad \mbox{ as } n\to \infty.
\end{equation}
\end{theorem}

\noindent The infinite derivatives of Takagi's function were first discussed by Begle and Ayres~\cite{Begle-Ayres} in 1936, and were finally characterized fully by Allaart and Kawamura \cite{AK} and Kr\"uppel \cite{Kruppel} in 2010. Compared with this result, our characterization of the infinite derivatives of $K(x)$  is surprisingly simple. As a Corollary of this Theorem, we also obtain yet another example of a measure zero set with Hausdorff dimension $1$.

\medskip

\section{Nowhere differentiability of $K(x)$} \label{sec:prelim}

Kobayashi mentioned that $F_a(x)$ has the following representation
\begin{equation}
\label{eq:kobayashi}
F_a(x)=\sum_{n=1}^{\infty}\prod_{l=1}^{n-1}p(\varepsilon_l(x))q(\varepsilon_n(x)),
\end{equation}
where 
$$p(0)=a,  \qquad p(1)=1-2a, \qquad p(2)=a,$$ 
$$q(0)=0, \qquad q(1)=a, \qquad q(2)=1-a.$$

\noindent Although Kobayashi did not give a proof, it is easy to see that \eqref{eq:kobayashi} satisfies the functional equation \eqref{FE:F_a}. Therefore, it is clear that $F_a(x)$ is an analytic function with respect to $a \in (0,1)$. 

Define the partial derivative of $F_a(x)$ with respect to $a$ as follows:
\begin{equation*}
\frac{\partial F_{a}(x)}{\partial a}:=\lim_{h \to 0}\frac{F_{a+h}(x)-F_a(x)}{h},
\end{equation*}
provided the limit exists. A simple calculation gives
\begin{equation*}
\frac{\partial F_{a}(x)}{\partial a}=s(\varepsilon_1(x))+\sum_{n=1}^{\infty} a^{n-I_1(1, n)}(1-2a)^{I_1(1, n)}q(\varepsilon_{n+1}(x)), 
\end{equation*} 
where $s(0)=0, s(1)=1$ and $s(2)=-1$. In particular, we have 
\begin{equation}
\label{expression-K1}
K(x)=\left.\frac{\partial F_{a}(x)}{\partial  a}\right|_{a=1/3}=\sum_{n=0}^{\infty}\frac{1}{3^n}\left\{s(\varepsilon_{n+1}(x))+(n-3I_1(1, n))\varepsilon_{n+1}(x)\right\}. 
\end{equation}

\noindent The graph of $K(x)$ is shown in Figure~\ref{figure1}. Note that $F_{1/3}(x)=x$.

%%%%%%%figure1%%%%%%%%%%%
\begin{figure}
\begin{center}
\includegraphics[height=5cm,width=5cm, bb=100 0 700 550]{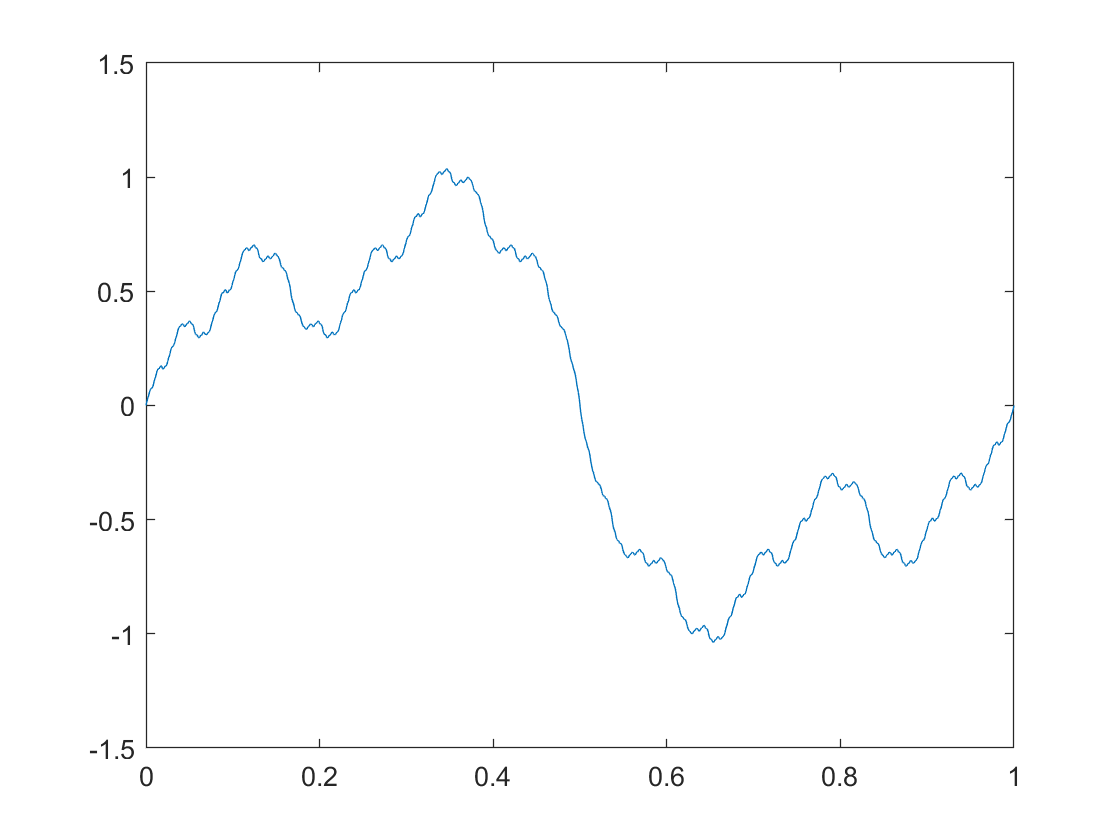}
\end{center}
\caption{Graph of $K(x):=\left.\frac{\partial F_{a}(x)}{\partial  a}\right|_{a=1/3}$} 
\label{figure1}
\end{figure}
%%%%%%%%%%%%%%%%%%%%%%%%%%%
\medskip

Another expression of $K(x)$ can be derived from \eqref{FE:F_a}. It is easy to check that $K(x)$ satisfies the following functional equation. 
\begin{equation}
\label{FE:K}
        K(x)=
        \begin{cases}
                K(3x)/3+3x,
                  & \qquad 0 \leq x \leq 1/3, \\
                K(3x-1)/3+3(1-2x),
                  & \qquad 1/3 \leq x \leq 2/3, \\
								K(3x-2)/3+3(x-1),
								  & \qquad 2/3 \leq x \leq 1, 
       \end{cases}
\end{equation}
where $K(0)=0=K(1)$. Here, we recall the following theorem of Yamaguti and Hata~\cite{Yamaguti-Hata1}:
\begin{theorem}[Yamaguti-Hata, 1983] \label{thm:Yamaguti-Hata}
Consider the functional equation: 
\begin{equation}
\label{eq:Yamagati-FE}
F(t,x)-tF(t, \psi(x))=g(x), 
\end{equation}
where $\psi(x)$ is a given one-dimensional dynamical system on $[0,1] \to [0,1]$ and $g(x)$ is a bounded measurable function on $[0,1]$. 
Then, there exist a unique bounded solution of \eqref{eq:Yamagati-FE} such that 
\begin{equation}
\label{FE: general}
F(t,x)=\sum_{n=0}^{\infty}t^n g(\psi^{(n)}(x)),
\end{equation}
where $\psi^{(n)}(x)$ is the $n-$fold iteration of $\psi$. 
\end{theorem}

\noindent Notice that the functional equation \eqref{FE:K} is a special case of \eqref{eq:Yamagati-FE}. More precisely, let $t=1/3$, and let
\begin{equation*}
\psi(x):=
\begin{cases}
	3x, & \qquad 0 \leq x \leq 1/3, \\
  3x-1, & \qquad 1/3 \leq x \leq 2/3, \\
	3x-2, & \qquad 2/3 \leq x \leq 1,
\end{cases}
\end{equation*}
and 
\begin{equation*}
g(x):=\Phi(x):=
\begin{cases}
	3x, & \qquad 0 \leq x \leq 1/3, \\
  3(1-2x), & \qquad 1/3 \leq x \leq 2/3, \\
	3(x-1), & \qquad 2/3 \leq x \leq 1.
\end{cases}
\end{equation*}

\noindent Then Theorem \ref{thm:Yamaguti-Hata} gives another expression for $K(x)$:
\begin{equation*}
K(x)=\sum_{n=0}^{\infty}\frac{1}{3^n}\Phi(3^{n}x),
\end{equation*}
where we extend $\Phi$ to all of $\mathbb{R}$ by $\Phi(x+1)=\Phi(x)$. Notice that this expression is different from \eqref{expression-K1}.

\begin{proposition}
\label{prop:K is odd}
$K(x)$ is an odd function.
\end{proposition}
\begin{proof}
Notice that $\Phi(x)=-\Phi(1-x)=\Phi(x+1)$. Thus, 
\begin{align*}
K(x)&=\sum_{n=0}^{\infty}\left(\frac13\right)^{n}\Phi(3^{n}x)=-\sum_{n=0}^{\infty}\left(\frac13\right)^{n}\Phi(1-3^{n}x) \\
		&=-\sum_{n=0}^{\infty}\left(\frac13\right)^{n}\Phi(3^{n}(1-x))=-K(1-x).
\end{align*}
\end{proof}

\begin{remark}
{\rm
A unique bounded solution of another special case of \eqref{FE: general} is Takagi's function. Comparing \eqref{expression-takagi} with \eqref{expression-K2}, $K(x)$ is somewhat similar to Takagi's function $T(x)$. Notice that $T(x)$ is based on the binary expansion of $x$ while $K(x)$ is based on the ternary expansion of $x$. Furthermore, $T(x)$ is even, but $K(x)$ is an odd function.  
}
\end{remark}

\begin{theorem}
\label{th:no-finite derivative}
The function $K$ is continuous, but it does not possess a finite derivative at any point.
\end{theorem}

\begin{proof}
First, define 
\begin{equation*}
K_n(x):=\sum_{k=0}^{n}\left(\frac13\right)^{k}\Phi(3^{k}x).
\end{equation*}
For each $n \in \NN$, it is clear that $K_n$ is continuous on $[0,1]$ and the sequence $(K_n)$ converges uniformly to $K$ since $|\Phi(3^{k}x)| \leq 1$. Thus, $K$ is continuous. 

Next, we modify the nowhere differentiability proof for Takagi's function by Billingsley~\cite{Billingsley}. Put $\Phi_k(x):=3^{-k}\Phi(3^k x)$ for $k=0,1,\dots$. Fix a point $x$, and, for each $n\in\NN$, let $u_n$ and $v_n$ be ternary rationals of order $n$ with $v_n-u_n=3^{-n}$ and $u_n\leq x<v_n$. Then
\begin{equation*}
\frac{K(v_n)-K(u_n)}{v_n-u_n}=\sum_{k=0}^{n-1}\frac{\Phi_k(v_n)-\Phi_k(u_n)}{v_n-u_n},
\end{equation*}
since $\Phi_k(u_n)=\Phi_k(v_n)=0$ for all $k\geq n$. But for $k<n$, $\Phi_k$ is linear on $[u_n,v_n]$ with slope $\Phi_k^+(x)$, the right-hand derivative of $\Phi_k$ at $x$. Thus,
\begin{equation*}
\frac{K(v_n)-K(u_n)}{v_n-u_n}=\sum_{k=0}^{n-1}\Phi_k^+(x).
\end{equation*}
Since $\Phi_k^+(x)=3$ or $-6$ for each $k$, this last sum cannot converge to a finite limit. Hence, $K$ does not have a finite derivative at $x$.
\end{proof}

\section{Improper infinite derivatives of $K(x)$}

Since $K(x)$ has no finite derivative at any point, it is natural to ask at which points $x \in [0,1]$ does $K(x)$ have an infinite derivative. These points are completely characterized by Theorem~\ref{main-theorem}. In order to prove our main Theorem, we use \eqref{expression-K2} as an expression for $K(x)$, and we define the right-hand and left-hand derivative of $K(x)$ by 
\begin{align*}
K_{+}^{'}(x):&=\lim_{h \to 0+}\frac{K(x+h)-K(x)}{h}, \\
K_{-}^{'}(x):&=\lim_{h \to 0-}\frac{K(x+h)-K(x)}{h},
\end{align*}
provided the limits exist.

\begin{proof}[Proof of Theorem~\ref{main-theorem}]
First we consider the right-hand derivative of $K(x)$. 
Let $x$ and $h$ be real numbers such that $0 \leq x <x+h<1$ and write 
\begin{equation*}
x=\sum_{k=1}^\infty \varepsilon_k/3^k, \qquad x+h=\sum_{k=1}^\infty \varepsilon_k'/3^k,
\end{equation*}
where $\varepsilon_k, \varepsilon_k' \in\{0,1,2\}$. As noted earlier, when $x$ is a triadic rational, there are two ternary expansions, but we choose the one which is eventually all zeros, except when $x=1$, in which case we take the expansion ending in all $2$'s. We adopt the same convention for $x+h$. 

\medskip

Let $p:=p(h) \in \NN$ such that $3^{-p} \leq h < 3^{-p+1}$. In other words, $p$ is the position of the first nonzero ternary digit of $h$. Let 
$$k_0:=\max \{k \in \NN: \varepsilon_{1}=\varepsilon_{1}',\varepsilon_{2}=\varepsilon_{2}', \dots \varepsilon_{k}=\varepsilon_{k}'\}.$$

\noindent Clearly, $0 \leq k_0 \leq p-1$. Observe that by the assumption for the expression of $x$, $k_0 \to \infty$ as $h \to 0+$. Set 
$$D_n(x,h):=\frac{\Phi(3^n(x+h))-\Phi(3^nx)}{3^nh}.$$
Notice that 
\begin{equation*}
3^nx=\sum_{k=1}^\infty \varepsilon_{n+k}/3^k, \qquad 3^{n}(x+h)=\sum_{k=1}^\infty \varepsilon_{n+k}'/3^k,
\end{equation*}
and $D_n(x,h)$ is the slope of the secant line between the two points $(3^{n}x, \Phi(3^{n}x))$ and $(3^{n}(x+h), \Phi(3^{n}(x+h)))$.

\noindent Then there are three cases to consider: (i) $k_0 \leq p-3$, (ii) $k_0=p-2$ and (iii) $k_0=p-1$.

(i) Assume first that $k_0 \leq p-3$. Then
\begin{equation*}
\frac{K(x+h)-K(x)}{h}=\Sigma_1+\Sigma_2+\Sigma_3+\Sigma_4,
\end{equation*}
where 
\begin{align*}
\Sigma_1:&=\sum_{n=0}^{k_0-1} D_n(x,h), \qquad \Sigma_2:=D_{k_0}(x,h), \\
\Sigma_3:&=\sum_{n=k_0+1}^{p-2} D_n(x,h), \qquad \Sigma_4:=\sum_{n=p-1}^{\infty} D_n(x,h). 
\end{align*}

\begin{itemize}
\item For $\Sigma_1$, since $\varepsilon_n=\varepsilon_n'$ for $1 \leq n \leq k_0$, 
\begin{equation*}
D_n(x,h)=3(-2)^{\varepsilon_{n+1}\!\!\!\!\mod 2} \mbox{ for } 0 \leq n \leq k_0-1. 
\end{equation*}

Define
\begin{align*}
I_i(a,b):&=\#\{j: a \leq j \leq b, \varepsilon_j=i\}, \\
f(a,b):=&f_x(a,b)=3I_0(a,b)-6I_1(a,b)+3I_2(a,b).
\end{align*}
Then we have 
\begin{equation*}
\Sigma_1=\sum_{n=0}^{k_0-1}D_n(x,h)=\sum_{n=1}^{k_0}3(-2)^{\varepsilon_{n}\!\!\!\!\mod 2}=f(1,k_0).
\end{equation*}

\item For $\Sigma_2:=D_{k_0}(x,h)$, since $\varepsilon_{k_0+1} \neq \varepsilon_{k_0+1}'$ and $D_{k_0}(x,h)$ is the slope of the secant line between the two points $(3^{k_0}x, \Phi(3^{k_0}x))$ and $(3^{k_0}(x+h), \Phi(3^{k_0}(x+h)))$, we have  
$$-6 \leq \Sigma_2 \leq 3.$$

\item For $\Sigma_3$, the assumption $k_0 \leq p-3$ implies that 
$$\varepsilon_n=2 \mbox{ and } \varepsilon_n'=0 \mbox{ for } k_0+2 \leq n \leq p-1.$$
For each $n$ with $k_0+2 \leq n \leq p-1$, there is an integer $m_n$ such that 
\begin{equation*}
3^nx \in [m_n-1/3, m_n) \mbox{ and } 3^n(x+h) \in [m_n, m_n+1/3).
\end{equation*}

Since $\Phi(x+1)=\Phi(x)$ and the $\Phi$ has the same slope of $3$ on both $[0, 1/3]$ and $[2/3, 1]$, we have 
\begin{equation*}
\Sigma_3=\sum_{n=k_0+1}^{p-2}D_n(x,h)=\sum_{n=k_0+2}^{p-1}3=3(p-k_0-2)=f(k_0+2,p-1).
\end{equation*}

\item For $\Sigma_4$, since $3^{-p} \leq h < 3^{-p+1}$ and $-1 \leq \Phi(x) \leq 1$, we may bound $\Sigma_4$ with a geometric series:
\begin{equation*}
 |\Sigma_4| \leq (2)3^p\sum_{n=p-1}^{\infty}\frac{1}{3^n}=9. 
\end{equation*}
\end{itemize}
\medskip

\noindent Therefore, if $k_0 \leq p-3$, we have 
\begin{equation*}
f(1,k_0)+f(k_0+2, p-1)-15 \leq \frac{K(x+h)-K(x)}{h} \leq f(1,k_0)+f(k_0+2, p-1)+12.
\end{equation*}
Since $-6 \leq f(k_0+1, k_0+1) \leq 3$, 
\begin{equation*}
f(1, p-1)-18 \leq \frac{K(x+h)-K(x)}{h} \leq f(1,p-1)+18.
\end{equation*}

\medskip

The other cases are much simpler. 

\noindent (ii) If $k_0=p-2$, notice that we may split them the sum in nearly the same way as in case (i), except we exclude $\Sigma_3$. So, 
\begin{equation*}
f(1,p-1)-18 \leq \frac{K(x+h)-K(x)}{h} \leq f(1,p-1)+18.
\end{equation*}

\noindent (iii) If $k_0=p-1$, we exclude $\Sigma_2$ and $\Sigma_3$. So, 
\begin{equation*}
f(1,p-1)-9 \leq \frac{K(x+h)-K(x)}{h} \leq f(1,p-1)+9.
\end{equation*}
This completes the preliminary analysis of the cases (i),(ii), and (iii). In any case,
\begin{equation*}
f(1,p-1)-18 \leq \frac{K(x+h)-K(x)}{h} \leq f(1,p-1)+18.
\end{equation*}

Next notice that as $h \to 0+$, $p \to \infty$. Therefore, it follows that  
\begin{equation*}
K_{+}^{'}(x)=\pm \infty \mbox{ iff } f(1,n) \to \pm \infty \mbox{ as } n\to \infty.
\end{equation*}

\noindent From Proposition~\ref{prop:K is odd}, we have 
\begin{equation*}
K_{-}^{'}(x)=K_{+}^{'}(1-x).
\end{equation*}
Notice that $f_x(1,n)\to \pm \infty \iff f_{1-x}(1,n)\to \pm \infty \mbox{ as } n\to \infty$. 
Therefore, the left-side derivative follows 
\begin{equation*}
K_{-}^{'}(x)=\pm \infty \mbox{ iff } f(1,n) \to \pm \infty \mbox{ as } n\to \infty.
\end{equation*}
This completes the proof.
\end{proof}

Define 
\begin{equation*}
p_1(x):=\lim_{n \to \infty} \frac{I_1(n)}{n},
\end{equation*}
assuming the limit exists, where $I_1(n)=\#\{j \leq n: \varepsilon_j=1\}$. 
Note that $p_1(x)$ is the frequency of the digit $1$ in the ternary expansion of $x$. 
We obtain the following immediate consequence of the main theorem. 

\begin{corollary}
\label{coro-1}
Assume $p_1(x)$ exists. Then
\begin{equation*}
\begin{cases}
K'(x)=+\infty, & \mbox{if }\quad p_1(x)< 1/3,\\
K'(x)=-\infty, & \mbox{if }\quad p_1(x)> 1/3. 
\end{cases}
\end{equation*}
\end{corollary}

\noindent It is well-known that almost all numbers (in the sense of Lebesgue measure) are {\it normal} to all bases; that is, they have base-$m$ expansions containing equal proportions of the digits $0, 1, \dots, m-1$ for all $m$. Notice that $p_1(x)=1/3$ is the boundary case in Corollary ~\ref{coro-1}.   

\begin{corollary}
Let $F:=\{x \in [0,1]: K'(x)=\pm \infty\}$. The Lebesgue measure of $F$ is $0$ while the Hausdorff dimension of $F$ is $1$. 
\end{corollary}

\begin{proof}
Let $I_1(n;x):=\#\{j\leq n: \varepsilon_j(x)=1\}$. We may view $I_1(n;x)$ as a random variable on the probability space $[0,1]$ equipped with Borel sets and Lebesgue measure. It is easy to see that the random sequence $(n-3I_1(n;x))_{n\in\mathbb{N}}$ is a mean zero random walk, so it oscillates infinitely often between positive and negative values with probability one. Thus, by Theorem \ref{main-theorem}, the Lebesgue measure of $F$ is $0$. 

\medskip
Let $p_0, p_1, p_2$ be real numbers of summing to $1$, so that $0<p_i<1$ and $\sum_{i=0}^{2}p_i=1$. Let $F(p_0, p_1, p_2)$ be the set of numbers $x \in [0,1)$ with ternary expansions containing the digits $0,1,2$ in proportions $p_0, p_1, p_2$ respectively. Let $(\alpha_n)$ be an increasing sequence converging to $1/3$ as $n \to \infty$. By Proposition 10.1 in Falconer's book~\cite{Falconer}, we have 
\begin{equation*}
\dim_{H}F\left(\frac{1-\alpha_n}{2}, \alpha_n, \frac{1-\alpha_n}{2}\right)=\frac{-(1-\alpha_n)\ln((1-\alpha_n)/2)-\alpha_n\ln \alpha_n}{\ln 3}, 
\end{equation*}
which converges to $1$ as $n \to \infty$. 

Since $\bigcup_{n} F(\frac{1-\alpha_n}{2}, \alpha_n, \frac{1-\alpha_n}{2}) \subset F$, the Hausdorff dimension of $F$ is 1. 
\end{proof}

\section*{Acknowledgment}
This undergraduate research was done mainly during summer 2021. We would like to thank Prof.~P.~Allaart, Prof.~J.~Iaia and Taylor Jones for their helpful comments and suggestions in preparing this paper.

\end{document}